\documentclass[12pt,leqno]{amsart}
\usepackage{amsmath}
\usepackage{color}
\usepackage{amssymb}
\def\overset#1#2{{\mathrel{\mathop {{#2}_{}}\limits^{#1}}}}
\def\underset#1#2{{\mathrel{\mathop {{}_{} {#2}}\limits_{{#1}_{}}}}}
\def\upplim_#1{\underset{#1}{\overline\lim}\;}
\def\lowlim_#1{\underset{#1}{\underline\lim}\;}
\newtheorem{corollary}[equation]{Corollary}

\newtheorem{lemma}[equation]{Lemma}
\newtheorem{proposition}[equation]{Proposition}
\newtheorem{remark}[equation]{\indent \rm {\it Remark}}
\newtheorem{theorem}[equation]{Theorem}

\newcommand{\C}{{\mathbb{C}}}

\renewcommand{\P}{{\mathbb{P}}}

\newcommand{\supp}{\mathrm{Supp}\,}

\newcommand{\Z}{\mathbb{Z}}

\numberwithin{equation}{section}

\title[A new degenerated second main theorem]{A new degenerated second main theorem for meromorphic mappings with hypersurfaces}

\author{Gerd Dethloff and Si Duc Quang}

\address{(Gerd Dethloff)
 Universit\'e de Bretagne Occidentale\\ 
    Laboratoire de math\'{e}matiques\\
UMR CNRS 6205\\
6, avenue Le Gorgeu, BP 452 \\
   29275 Brest Cedex, France}
\email{Email: gerd.dethloff@univ-brest.fr}

\address{(Si Duc Quang)  Department of Mathematics \\
Hanoi National University of Education\\
136-Xuan Thuy, Cau Giay, Hanoi, Vietnam}
\email{Email: quangsd@hnue.edu.vn}

\begin{document}

\begin{abstract}
We establish a second main theorem with truncated counting functions for algebraically nondegenerate meromorphic mappings into a projective variety and a family of hypersurfaces in subgeneral position. The above bound of the total defect obtained from our result is better than that of the previous results. Moreover, in our result, the truncation level of the counting functions is estimated explicitly and independently of the number of hypersurfaces. Especially, we do not need the assumption that the family of hypersurfaces must satisfy the Bezout properties as imposed in some earlier studies, but only a weak Bezout property.
\end{abstract}

\def\thefootnote{\empty}

\maketitle

\section{Introduction}

\vskip0.2cm 

Let $V$ be a $k$-dimension subvariety of $\P^n(\C)$ and $\mathcal Q=\{Q_1,\ldots,Q_q\}$ a family $q$ hypersurfaces in $\P^n(\C)$. The family $\mathcal Q$ is said to be in $N$-subgeneral position $(q\ge N+1)$ with respect to $V$ if
$$ V\cap\bigcap_{j=0}^NQ_{i_j}=\emptyset\ \forall 1\le i_0<\cdots<i_{N}\le q.$$
If $N=k$, the family $\mathcal Q$ is said to be in general position with respect to $V$.

Let $f$ be a meromorphic mapping from $\C^m$ into $\P^n(\C)$. For each hypersurface $Q$ in $\P^n(\C)$ with $f(\C)\not\subset Q$, we denote by $\nu_{Q(f)}$ the pull-back of the divisor $Q$ by $f$. As usual, we denote by $T_f(r)$ the characteristic function of $f$ with respect to the hyperplane line bundle of $\P^n(\C)$ and $N^{[M]}_{Q(f)}(r)$ the counting function of $\nu_{Q(f)}$ with multiplicities truncated to level $M$ (see Section 2 for the definitions). The defect of $f$ with respect to $Q$ is defined by
$$ \delta^{[M]}_f(Q) =1-\overline{\lim\limits_{r\rightarrow\infty}}\frac{N^{[M]}_{Q(f)}(r)}{dT_f(r)}.$$
  
In 1933, H. Cartan \cite{Ca} established a second main theorem (SMT) for linearly nondegenerate meromorphic mappings and hyperplanes (i.e., hypersurfaces of degree 1) as follows.

\vskip0.2cm
\noindent
\textbf{Theorem A.} {\it Let $f:  {\C}^m \to \P^n(\C)$ be a linearly nondegenerate meromorphic mapping and $\{H_i\}_{i=1}^q$ be hyperplanes in general position in $\P^n(\C)$. Then we have}
$$\|\ (q-n-1)T_f(r) \leq \sum_{i=1}^q N^{[n]}_{H_i(f)}(r)+ o(T_f(r)).$$
Here, by the notation ``$\| P$'' we mean that the assertion $P$ holds for all $r\in [0,\infty)$ excluding a Borel subset $E$ of the interval $[0,\infty)$ with $\int_E dr<\infty$. In the above theorem, the truncation level is $n$ and the total defect is bounded above by $n+1$, i.e., $\sum_{i=1}^q\delta^{[n]}_{H_i(f)}\le n+1$.

If the map $f$ is linearly degenerate, we may regard $f$ as a linearly nondegenerate mapping into the smallest subspace $V$ of $\P^n(\C)$ containing the image $f(\C)$. However in that case, the family of hyperplanes $\{H_i\}_{i=1}^q$ may not be in general position with respect to $V$, but in subgeneral position. By introducing the notion of Nochka weight, in 1983 Nochka \cite{Noc83} gave the following degenerated second main theorem for meromorphic mappings with hyperplanes as follows.

\vskip0.2cm
\noindent
\textbf{Theorem B} (cf. \cite{Noc83}). {\it Let $f:  {\C}^m \to \P^n(\C)$ be a linearly nondegenerate meromorphic mapping and $\{H_i\}_{i=1}^q$ a  family of  hyperplanes in $N$-subgeneral position in $\P^n(\C)$ $(N\ge n)$ such that $f(\C^m)\not\subset H_i\ (1\le i\le q)$. Then we have
$$\|\ \ (q-2N+n-1)T_f(r) \leq \sum_{i=1}^q N^{[n]}_{H_i(f)}(r)+ o(T_f(r)).$$}
In this SMT, the truncation level is $n$ and the total defect is bounded above by $2N-n+1$.

Over the last few decades, there have been several results generalizing this theorem to the case where the hyperplanes are replaced by hypersurfaces. The first SMT for hypersurfaces is given by A. E. Eremenko and M. L. Sodin \cite{ES}. They considered an arbitrary holomorphic curve $f$ from $\C$ into $\P^N(\C)$ with $q$ hypersurfaces $\{Q_i\}_{i=1}^q$ in general position and shown that
$$\|\  (q-2N-\epsilon)T_f(r) \leq \sum_{i=1}^q \dfrac{1}{\deg Q_i}N(r,f^*Q_i),$$
for every $\epsilon >0$. Then, from this SMT, the total defect is bounded above by $2N$, but the truncation level is lost.
 
Recently, motivated by the method from Diophantine approximation, many authors have established the SMTs for meromorphic mappings and hypersurfaces in subgeneral position. Firstly, consider the case where $f:\C^m\rightarrow\P^n(\C)$ is an algebraically nondegenerate meromorphic mapping and $\mathcal Q=\{Q_1,\ldots,Q_q\}$ is a family of hypersurfaces not containing the image of $f$. By using the filtration method of P. Corvaja and U. Zannier \cite{CZ}, some authors have given some SMTs for such mapping $f$ with the family $\mathcal Q$. We list here the names of them and the above bounds of the total defect they obtained:

$\bullet$ M. Ru \cite{Ru04}: Assume that $\mathcal Q$ is in general position in $\P^n(\C)$. The total defect is bounded above by $n+1$, the truncation level is lost.

$\bullet$ T. T. H. An-H. T. Phuong \cite{AP}: Assume that $\mathcal Q$ is in general position in $\P^n(\C)$. The total defect is bounded above by $n+1$ and the truncation level is $M= 2d\left\lceil 2^n(n+1) n(d+1) \varepsilon^{-1}\right\rceil^n$, where $d=lcm(\deg Q_1,\ldots,\deg Q_q)$ and  $\lceil x\rceil$ stands for the smallest integer not less than the real number $x$. Also, by $[x]$ we denote the largest integer not exceeding $x$. 

Later on, motivated by the method of J. Evertse and R. Ferretti in studying on the Schmidt's subspace theorem, M. Ru \cite{Ru09} initially established the SMT for algebraically nondegenerate holomorphic curves into a projective subvariety $V\subset\P^n(\C)$ of dimension $k$ with a family of hypersurfaces $\mathcal Q=\{Q_i\}_{i=1}^q$ in general position with respect to $V$. The total defect in his result is bounded above by $k+1$, but the truncation level is not considered. By developing the method of M. Ru, many authors have given second main theorems for the case where $\mathcal Q$ in $N$-subgeneral position with respect to $V$. We list here the names of them with their above bounds of the total defect:
\begin{itemize}
\item Z. Chen, M. Ru and Q. Yan \cite{CRY}; the total defect is bounded by $N(k+1)$ and truncation level is not considered.
\item L. Shi and M. Ru \cite{LR}; the total defect is bounded above by $\frac{N(N-1)(k+1)}{N+k-2}$ and truncation level is not estimated.
\item L. Giang \cite{LG}; the total defect is bounded above by $N(k+1)+\epsilon$ and the truncation level is $L=\lceil e^kd^{k^2+k}\times([7 k N(q+1)\epsilon^{-1}]+2)^k\deg(V)^{k+1}\rceil$, where $d=lcm(\deg Q_1,\ldots,\deg Q_q)$.
\end{itemize}
The most optimal bound above of the total defect for the case of hypersurfaces in subgeneral position with respect to $V$ is given by the second named author as follows:
\begin{itemize}
\item S. D. Quang \cite{Q19}; the total defect is bounded above by $(N-k+1)(k+1)+\epsilon$ and the truncation level is $$L=\left[\deg(V)^{k+1}e^kd^{k^2+k}(N-k+1)^k(2k+4)^k(k+1)^k(q!)^k \epsilon^{-k}\right].$$
\end{itemize}
Recently, in \cite{Q22a,Q22c} the second named author defined the notion of the distributive constant of $\mathcal Q$ with respect to $V$ (cf. \cite[Definition 3.3]{Q22a}) as follows:
$$ \Delta_{\mathcal Q,V}:=\underset{\emptyset\ne\Gamma\subset\{1,\ldots,q\}}\max\dfrac{\sharp\Gamma}{\dim V-\dim\left (V\cap\bigcap_{j\in\Gamma} Q_j\right )}.$$
(note that $\dim\emptyset =-\infty$). For $V=\P^n(\C)$, we will write $\Delta_{\mathcal Q}$ for $\Delta_{\mathcal Q,\P^n(\C)}$. Using this notion, Quang has proved the following SMT for an arbitrary family $\mathcal Q$ of hypersurfaces: 
$$\bigl \|\ \left (q-\Delta_{\mathcal Q,V}(k+1)-\epsilon\right) T_f(r)\le\sum_{i=1}^q\frac{1}{\deg Q_i}N^{[M]}_{Q_i(f)}(r)+o(T_f(r)),$$
where $M=\left[d^{k^2+k}\deg (V)^{k+1}e^k\Delta_{\mathcal Q,V}^k(2k+4)^k(k+1)^k(q!)^k\epsilon^{-k}\right]$. 

By simple computation, we have that $\Delta_{\mathcal Q,V}\le (N-k+1)$ (see \cite[p. 168]{Q22a}), but the inequality may become the equality in many cases. Therefore, although the above result covers most previous results for the case of  hypersurfaces in subgeneral position, in the case of an arbitrary family of hyperplanes in $N$-subgeneral position, we are still unable to replace the constant $(N-k+1)$ by a smaller constant which depends only on $N$ and $k$. Recently, S. Lei and Q. Yan \cite{LQ} gave a refinement of the above result of Quang. But they need the family of hypersurfaces satisfying the Bezout properties
$$\operatorname{codim}_V\bigl(\bigcap_{j\in I\cup J}Q_j\bigl) \leq \operatorname{codim}_V\bigl(\bigcap_{j\in I}Q_j\bigl) + \operatorname{codim}_V\bigl(\bigcap_{j\in J}Q_j\bigl)$$
for every $I, J\subset\{1, \ldots, q\},$ and if the family of hypersurfaces is in $N$-general position then they obtained that 
$$\|\ \bigl(q - \min \bigl\{N-k, \frac{N-k}{2} + 1\bigl\} - p(k+1) - \varepsilon\bigl) T_{f}(r)\leq \sum_{j=1}^q \frac{1}{\deg Q_j} N^{[M_0]}(r, f^*Q_j)$$
where $p = \frac{N-k+2}{2} + \frac{N-k}{2k}$ and $M_0 = \left[ (\deg \varphi(V))^{k+1} e^k \Delta^k (2k+5)^k (k+1)^k (q!)^k \varepsilon^{-k} \right]$. Then, we may see that the defect obtained is still far from the expectation $(2N-k+1)$. Hence, it is believed that the above bound of the total defect of the SMT for the case of subgeneral hypersurfaces can be reduced. Also, we see that the truncation levels given in the above mentioned results (of Quang \cite{Q19,Q22a}, Giang \cite{LG} and Lei-Yan \cite{LQ}) for this case depend on the number $q$ of involving hypersurfaces. Then, their application may be restricted (eg. for the ramification problem or the uniqueness problem of meromorphic mappings). 

Our purpose in this paper is to improve the SMT for meromorphic mappings with families of hypersurfaces in subgeneral position with respect to a subvariety $V\subset\P^n(\C)$. Our improvements include reducing both the previously known bound on the total defect and the truncation level, as well as weakening the assumption on the Bezout property. In particular, the truncation level in our result is explicitly estimated and is independently of the number of hypersurfaces involved. Specifically, we establish the following theorem.

 \begin{theorem}\label{1.1} 
Let $V$ be a $k$-dimension projective subvariety of  $\P^n(\C)$. Let $f$ be an algebraically nondegenerate meromorphic mapping from $\C^m$ into $V$. Let $\mathcal Q=\{Q_1,\ldots,Q_q\}$ be a family of hypersurfaces of $\P^n(\C)$ in $N$-subgeneral position with respect to $V$ such that $f(\C^m)\not\subset Q_i\ (1\le i\le q)$, and $d$ be the least common multiple of $\deg Q_1,\ldots,\deg Q_q$. Assume that the family $\mathcal Q$ satisfies the following weak Bezout property: $\operatorname{codim}_V\bigl(\bigcap_{j\in I\cup J}Q_j\bigl) \leq 2$ for every $I,J\subset\{1,\ldots,q\}$ such that $\operatorname{codim}_V\bigl(\bigcap_{j\in I}Q_j\bigl)=\operatorname{codim}_V\bigl(\bigcap_{j\in J}Q_j\bigl)=1.$ Then, for every $\epsilon >0$, 
$$\biggl \|\ \left (q-\left[\lambda+\frac{1}{2}\right]-\lambda(k+1)-\epsilon\right) T_f(r)\le\sum_{i=1}^q\frac{1}{\deg Q_i}N^{[L]}(r,f^*Q_i)+o(T_f(r)),$$
where $\lambda=\frac{N-k+2}{2}$ and $L=\lceil d^{k^2+k}\deg(V)^{k+1}e^k(2k+5)^k(\lambda^2(k+1)\epsilon^{-1}+\lambda)^k-2\rceil$.
\end{theorem}

 It is clear that $\left[\lambda+\frac{1}{2}\right]+\lambda(k+1)$ is smaller than $\Delta_{\mathcal Q,V}$ in many cases. Compared with the result of Lei-Yan, our upper bound for the total defect is smaller (lacking the component $\frac{N-k}{2k}(k+1)$), and the assumption of the weak Bezout property is weaker than the full Bezout property. Moreover, we note that, while Lei and Yan developed the method of \cite{Q22a} by introducing the rather complicated notion of the distributive pair, our approach is simpler: we find a way to disregard all hypersurfaces whose intersections with other hypersurfaces are ``too large''. We also see that the truncation level $L$ in our result do not depend on the number $q$. In order to do so, we employ the new lower bound for the Chow weight given in \cite{Q22b} (see Lemma \ref{2.4}) and introduce new techniques to control the error term arising from the application of the Hilbert weight estimate (Theorem \ref{2.3}).

The meromorphic mapping $f$ from $\C^m$ into $V\subset\P^n(\C)$ is said to be ramified over a hypersurface $Q$ with multiplicity at least $s$ if either $f(\C^m)\subset Q$ or $\nu_{Q(f)}(z)\ge s$ for all $z\in\supp\nu_{Q(f)}$. From Theorem \ref{1.1}, by choosing $\epsilon =\frac{1}{2}$, we immediately get the following corollary.

\begin{corollary}\label{1.2}
Let $V\subset\P^n(\C)$ be a projective subvariety of dimension $k\ge 1$. Let $\{Q_1,\ldots,Q_q\}$ be a family of hypersurfaces of $\P^n(\C)$ in $N$-subgeneral position with respect to $V$. Let $f$ be a meromorphic mapping from $\C^m$ into $V$, which is ramified over each $Q_i$ with multiplicity at least $s_i\ (1\le i\le q)$ such that
$$\sum_{i=1}^q\left(1-\frac{L}{s_i}\right)>\lambda(k+2)+1,$$
where $\lambda=\frac{N-k+2}{2}$ and $L=\lceil d^{k^2+k}\deg(V)^{k+1}e^k(2k+5)^k(2\lambda^2(k+1)+\lambda)^k-2\rceil$. Then $f$ is algebraically degenerate.
\end{corollary}

\noindent
\textbf{Remark. }The SMT for moving hypersurfaces also has been studied intensively by many authors, such as G. Dethloff-T. V. Tan \cite{DT11,DT20}, S. D. Quang \cite{Q18,Q22c}. However, the truncation level in these results are very large and depend on the number of hyperpsurfaces.
 
 \noindent
{\bf Acknowledgements.} This work was done during a stay of the second named author at Universit\'{e} de Bretagne Occidentale with the support of CNRS. He would like to thank the university for their kind hospitality he received and thank CNRS for their support. The research of the second author is funded by Vietnam National Foundation for Science and Technology Development (NAFOSTED) under grant number 101.02-2021.12. 

\section{Basic notions and auxiliary results from Nevanlinna theory}

\noindent
{\bf A. Notation.}\ Set $\|z\| = \big(|z_1|^2 + \dots + |z_m|^2\big)^{1/2}$ for
$z = (z_1,\dots,z_m) \in \C^m$ and define $B(r) := \{ z \in \C^m : \|z\| < r\},\quad S(r) := \{ z \in \C : \|z\| = r\}\ (0<r<\infty).$
Define
$$v_{m-1}(z) := \big(dd^c \|z\|^2\big)^{m-1}\quad \quad \text{and}$$
$$\sigma_m(z):= d^c \log\|z\|^2 \land \big(dd^c\log\|z\|^2\big)^{m-1}
 \text{on} \quad \C \setminus \{0\}.$$

Let $f : \mathbb C^m \to \P^n(\C)$ be a meromorphic mapping with a reduced representation $\tilde f = ( f_0,\ldots, f_n)$. The characteristic function of $f$ is defined by
$$ T_f (r)=\int_1^{r}\frac{dt}{t}\int_{B(t)}f^*\Omega\wedge v_{m-1},$$
where $\Omega$ is the Fubini-Study form on $\P^n(\C)$. By Jensen's formula, we have
$$ T_f(r)=\int_{S(r)}\log \|\tilde f\|\sigma_m-\int_{S(1)}\log \|\tilde f\|\sigma_m,$$
where $\|\tilde f\|=(|f_0|^2+\cdots+ |f_n|^2)^{\frac{1}{2}}$.

\vskip0.2cm 
\noindent

Now for a divisor $\nu$ on $\mathbb C^m$ and for a positive integer $M$ or $M= \infty$, we set
$$\nu^{[M]}(z)=\min\ \{M,\nu(z)\},$$
and define, for $1<r<\infty$, 
\begin{align*}
N(r,\nu)=\int\limits_1^r \dfrac {n(t)}{t^{2m-1}}dt\ \text{ where }\ 
n(t) =
\begin{cases}
\int\limits_{|\nu|\,\cap B(t)}
\nu(z) \alpha^{m-1} & \text  { if } m \geq 2,\\
\sum\limits_{|z|\leq t} \nu (z) & \text { if }  m=1. 
\end{cases}
\end{align*}
Similarly, we define $n^{[M]}(t)$ and $N(r,\nu^{[M]})$ (denote it simply by $N^{[M]}(r,\nu)$).

For a nonzero meromorphic function $\varphi$ on $\C^m$, we denote by $\nu^0_{\varphi}$ (resp. $\nu^\infty_\varphi$) the zero divisor (resp. pole divisor) of $\varphi$. We write $N_{\varphi}(r)$ (resp. $N^{[M]}_{\varphi}(r)$) for the counting function $N(r,\nu^0_\varphi)$ (resp. $N^{[M]}(r,\nu^0_\varphi)$).

In this paper, for a hypersurface $Q$ of degree $d$ in $\P^n(\C)$, if there is no confusion, we use the same notation $Q$ to denote its defining homogeneous polynomial, i.e., 
$$ Q(x_0,\ldots,x_n)=\sum_{I\in\mathcal T_d}a_Ix^I,\ a_I\in\C, $$ 
where $\mathcal T_d=\{(i_0,\ldots,i_n)\in\mathbb Z^{n+1}_{\ge 0}; i_0+\cdots+i_n=d\}$, $x^I=x_0^{i_0}\cdots x_n^{i_n}$ for each $I=(i_0,\ldots,i_n)\in\mathcal T_d$. The proximity function of $f$ with respect to $Q$, denoted by $m_f (r,Q)$, is defined by
$$m_f (r,Q)=\int_{S(r)}\log\frac{\|\tilde f\|^d}{|Q(\tilde f)|}\sigma_m-\int_{S(1)}\log\frac{\|\tilde f\|^d}{|Q(\tilde f)|}\sigma_m,$$
where $Q(\tilde f)=Q(f_0,...,f_n)$. This definition is independent from the choice of the reduced representation of $f$. 

Since the counting function $N^{[M]}_{Q(\tilde f)}(r)$ does not depend on the choice of the reduced representation $\tilde f$, we just denote it by $N^{[M]}_{Q(f)}(r)$. By Jensen's formula, we have
$$N_{Q(f)}(r)=\int_{S(r)}\log |Q(\tilde f)|\sigma_m-\int_{S(1)}\log |Q(\tilde f)| \sigma_m.$$
The first main theorem in Nevanlinna theory for meromorphic mappings and hypersurfaces is stated as follows.
$$dT_f (r)=m_f (r,Q) + N_{Q(f)}(r)+O(1).$$

\vskip0.2cm 
\noindent
{\bf B. Auxiliary results.} For each meromorphic function $g$ on $\C^m$ and  an $m$-tuples $\alpha =(\alpha_1,...,\alpha_m)\in\Z^m_{\ge 0}$, we set $|\alpha|=\sum_{i=1}^m\alpha_i$ and define 
$${\mathcal D}^{\alpha}g=\frac{\partial^{|\alpha|}g}{\partial^{\alpha_{1}}z_1...\partial^{\alpha_{m}}z_m}.$$

We have the lemma on logarithmic derivative stated as follows.
\begin{lemma}[{see \cite{NO}}]\label{2.1}
Let $f$ be a nonzero meromorphic function on $\C^m.$ Then for all positive integer $k$,
$$\biggl\|\quad m\biggl(r,\dfrac{\mathcal D^{\alpha}(f)}{f}\biggl)=o(T_f(r))\ (\alpha\in \mathbb Z^m_+).$$
\end{lemma}

Repeating the argument in (\cite{F85}, Proposition 4.5), we have the following.

\begin{proposition}\label{2.2}
Let $F_0,\ldots ,F_{N}$ be meromorphic functions on $\mathbb C^m$ such that $\{F_0,\ldots ,F_{N}\}$ are  linearly independent over $\mathbb C.$ Then  there exists an admissible set  
$$\alpha=\{\alpha_i=(\alpha_{i1},\ldots,\alpha_{im})\}_{i=0}^{N} \subset \mathbb Z^m_{\ge 0}$$
with $|\alpha_i|=\sum_{j=1}^{m}|\alpha_{ij}|\le i \ (0\le i \le N)$ such that the following are satisfied:

(i)\  $W^{\alpha}(F_0,\ldots ,F_{N})\overset{Def}{:=}\det{({\mathcal D}^{\alpha_i}\ F_j)_{0\le i,j\le N}}\not\equiv 0.$ 

(ii) $W^{\alpha}(hF_0,\ldots ,hF_{N})=h^{N+1}W^{\alpha}(F_0,\ldots ,F_{N})$ for every nonzero meromorphic function $h$ on $\C^m.$

(iii) $W^{\alpha}(L_0(F_0,\ldots ,F_{N}),\ldots,L_N(F_0,\ldots ,F_{N}))=cW^{\alpha}(F_0,\ldots,F_{N})$ for any $N+1$ independent linear forms $L_0,\ldots,L_N$ of $N+1$ variables, where $c$ is a nonzero constant.
\end{proposition}

\vskip0.2cm 
\noindent

Let $X\subset\P^n(\C)$ be a projective subvariety of dimension $k$ and degree $\delta$. For $\textbf{a} = (a_0,\ldots,a_n)\in\mathbb Z^{n+1}_{\ge 0}$ we write ${\bf x}^{\bf a}$ for the monomial $x^{a_0}_0\cdots x^{a_n}_n$. Let $I=I_X$ be the prime ideal in $\C[x_0,\ldots,x_n]$ defining $X$. Let $\C[x_0,\ldots,x_n]_m$ denote the vector space of homogeneous polynomials in $\C[x_0,\ldots,x_n]$ of degree $m$ (including $0$). For $m = 1, 2,\ldots,$ put $I_m :=\C[x_0,\ldots,x_n]_m\cap I$ and define the Hilbert function $H_X$ of $X$ by
\begin{align*}
H_X(m):=\dim (\C[x_0,\ldots,x_n]_m/I_m).
\end{align*}
Let ${\bf c}=(c_0,\ldots,c_n)$ be a tuple in $\mathbb R^{n+1}_{\ge 0}$ and let $e_X({\bf c})$ be the Chow weight of $X$ with respect to ${\bf c}$. For ${\bf a}$ as above, we set ${\bf a}\cdot{\bf c}=\sum_{i=0}^na_ic_i$. The $u$-th Hilbert weight $S_X(u,{\bf c})$ of $X$ with respect to ${\bf c}$ is defined by
\begin{align*}
S_X(u,{\bf c}):=\max\left (\sum_{i=1}^{H_X(u)}{\bf a}_i\cdot{\bf c}\right),
\end{align*}
where the maximum is taken over all sets of monomials ${\bf x}^{{\bf a}_1},\ldots,{\bf x}^{{\bf a}_{H_X(u)}}$ whose residue classes modulo $I$ form a basis of $\C[x_0,\ldots,x_n]_u/I_u.$

The following theorem is due to J. Evertse and R. Ferretti \cite{EF1,EF2}.
\begin{theorem}[{Theorem 4.1 \cite{EF1}}]\label{2.3}
Let $X\subset\P^n(\C)$ be an algebraic variety of dimension $k$ and degree $\delta$. Let $u>\delta$ be an integer and let ${\bf c}=(c_0,\ldots,c_n)\in\mathbb R^{n+1}_{\geqslant 0}$.
Then
$$ \dfrac{1}{uH_X(u)}S_X(u,{\bf c})\ge\dfrac{1}{(k+1)\delta}e_X({\bf c})-\dfrac{(2k+1)\delta}{u}\cdot\left (\max_{i=0,\ldots,n}c_i\right). $$
\end{theorem}

The following lemma is proved by the second named author \cite{Q22b} for the case of number fields, but it automatically holds for the case of the complex field.
\begin{lemma}[{see \cite[Lemma 3.2]{Q22b}}]\label{2.4}
Let $Y$ be a projective subvariety of $\P^R(\C)$ of dimension $k\ge 1$ and degree $\delta_Y$. Let $\ell\ (\ell\ge k+1)$ be an integer and let ${\bf c}=(c_0,\ldots,c_R)$ be a tuple of non-negative reals. Let $\{H_0,\ldots,H_R\}$ be a set of hyperplanes in $\P^R(\C)$ defined by $H_{i}=\{y_{i}=0\}\ (0\le i\le R)$. Let $\{i_1,\ldots, i_\ell\}$ be a subset of $\{0,\ldots,R\}$ such that:
\begin{itemize}
\item[(1)] $c_{i_\ell}=\min\{c_{i_0},\ldots,c_{i_\ell}\}$,
\item[(2)] $Y\cap\bigcap_{j=1}^{\ell}H_{i_j}= \emptyset$, 
\item[(3)] and $Y\not\subset H_{i_j}$ for all $j=1,\ldots,\ell$.
\end{itemize}
Let $\Delta_{\mathcal H,Y}$ be the distributive constant of the family $\mathcal H=\{H_{i_j}\}_{j=1}^\ell$ with respect to $Y$. Then
$$e_Y({\bf c})\ge \frac{\delta_Y}{\Delta_{\mathcal H,Y}}(c_{i_1}+\cdots+c_{i_\ell}).$$
\end{lemma}

\section{Proof of Theorem \ref{1.1}}
 In order to prove Theorem \ref{1.1}, we first prove the following lemmas.
\begin{lemma}\label{new}
Let $V$ be an $k$-dimension subvariety of $\P^n(\C)\ (n\ge k)$ and $\mathcal Q=\{Q_1,\ldots,Q_q\}$ a family of hypersurfaces in $\P^n(\C)$ located in $N$-subgeneral position with respect to $V\ (N> k)$ such that $V\not\subset Q_i\ (1\le i\le q)$. Let $\ell=\lceil \frac{N-k+3}{2}\rceil$. Assume that $\dim V\cap\bigcap_{j=1}^\ell Q_{i_j}\le k-2$ for every $1\le j_1<\cdots<j_\ell\le q$. Then $\Delta_{\mathcal Q,V}\le\frac{N-k+2}{2}$.
\end{lemma}
\begin{proof} 
Let $\Gamma$ be a nonempty set of $\{1,\ldots,q\}$ such that $V\cap\bigcap_{j\in\Gamma}Q_{j}\ne\emptyset$. We distinguish the following three cases:\\
Case 1: $\sharp\Gamma\le\frac{N-k+2}{2}$. It is clear that $\frac{\sharp\Gamma}{k-\dim V\cap\bigcap_{j\in\Gamma}Q_{j}}\le \frac{N-k+2}{2}.$\\
Case 2: $\sharp\Gamma\in (\frac{N-k+2}{2},N-k+2]$. By the assumption, we have $\frac{\sharp\Gamma}{k-\dim V\cap\bigcap_{j\in\Gamma}Q_{j}}\le\frac{N-k+2}{2}.$\\
Case 3: $\sharp\Gamma> N-k+2.$ Suppose that $\Gamma=\{1,\ldots,t\}$, where $t\ge N-k+3$. Since $V\cap\bigcap_{j=1}^{N+1}Q_{j}=\emptyset$, it is easy to see that $t\le N$ and $ \dim V\cap\bigcap_{j=1}^{t}Q_{j}\le N-t$. Therefore,
$$ \frac{\sharp\Gamma}{k-\dim V\cap\bigcap_{j\in\Gamma}Q_{j}}\le\frac{t}{k-(N-t)}\le\frac{N-k+3}{3}\le\frac{N-k+2}{2}.$$

Then, for every subset $\Gamma\subset\{1,\ldots,q\}$, we have $\frac{\sharp\Gamma}{k-\dim V\cap\bigcap_{j\in\Gamma}Q_{j}}\le\frac{N-k+2}{2}$. This yields that $\Delta_{\mathcal Q,V}\le\frac{N-k+2}{2}$.
\end{proof}

\begin{lemma}\label{new2} Let $V$ be an $k$-dimension subvariety of $\P^n(\C)$. Let $\mathcal Q=\{Q_1,\ldots,Q_q\}$ be a family of $q$ hypersurfaces of $\P^n(\C)$ in $N$-subgeneral position with respect to $V$. Assume that the family $\mathcal Q$ satisfies the following weak Bezout property: $\operatorname{codim}_V\bigl(\bigcap_{j\in I\cup J}Q_j\bigl) \leq 2$ for every $I,J\subset\{1,\ldots,q\}$ such that $\operatorname{codim}_V\bigl(\bigcap_{j\in I}Q_j\bigl)=\operatorname{codim}_V\bigl(\bigcap_{j\in J}Q_j\bigl)=1.$ Let $\ell=\lceil \frac{N-k+3}{2}\rceil$. Then there is a subset $\Gamma$ of $\{1,\ldots,q\}$ with $\sharp\Gamma\ge q-N+k-3+\ell$ such that the family of hypersurfaces $\mathcal P=\{Q_i;i\in\Gamma\}$ satisfies $\Delta_{\mathcal P,V}\le\frac{N-k+2}{2}$.
\end{lemma}
\begin{proof}
Let $p_1$ be the biggest integer such that there exists a subset $\Gamma_1\subset\{1,\ldots,q\}$ with $\sharp\Gamma_1=p_1$ and $\dim V\cap\bigcap_{j\in\Gamma_1}Q_j=k-1$. Let $p_2$ be the biggest integer such that there exists a subset $\Gamma_2\subset\{1,\ldots,q\}\setminus\Gamma_1$ with $\sharp\Gamma_1=p_2$ and $\dim V\cap\bigcap_{j\in\Gamma_2}Q_j=k-1$ by the weak Bezout property. Therefore $\dim V\cap\bigcap_{j\in\Gamma_1\cup\Gamma_2}Q_j\ge k-2$. This implies that $p_1+p_2\le N-k+2$ and $p_1\ge p_2$. Then $p_2\le\frac{N-k+2}{2}<\ell$.

Take $\Gamma_1'$ a subset of $\Gamma_1$ with $\sharp\Gamma_1'=(\ell-1)-p_2$ and set $\Gamma=\{1,\ldots,q\}\setminus(\Gamma_1\setminus\Gamma_1'), p=\sharp\Gamma$ and $\mathcal P=\{Q_j;\ j\in\Gamma\}.$ We may suppose that $\mathcal P=\{Q_1,\ldots,Q_p\}$ with 
\begin{align*}
p&=q-(p_1+p_2-(\ell-1))\ge q-(N-k+2)+\ell-1\\
&=q-N+k-3+\ell.
\end{align*}
Then $\mathcal P$ is a family of $p$ hypersurfaces in $N$-subgeneral position with respect to $V$ satisfying
$$\dim V\cap\bigcap_{s=1}^{\ell}Q_{j_s}\le k-2 \text{ for every }1\le j_1<\cdots<j_\ell\le p.$$ 
By Lemma \ref{new}, we have $\Delta_{\mathcal P,V}\le\frac{N-k+2}{2}$.
\end{proof} 

\begin{proof}[{\sc Proof of Theorem \ref{1.1}}]
Let $\ell=\lceil \frac{N-k+3}{2}\rceil$. By Lemma \ref{new2}, there exists a subset $\mathcal P$ of $\mathcal Q$ with $p=\sharp\mathcal P\ge q-N+k-3+\ell\ge q-\left[\frac{N-k+3}{2}\right]$ and $\Delta_{\mathcal P,V}\le\frac{N-k+2}{2}=\lambda$.
 
Without loss of generality, we suppose that $\mathcal P=\{Q_1,\ldots,Q_p\}$ and $p>\Delta_{\mathcal P,V}(k+1)$. It implies that $\bigcap_{\underset{j\ne i}{1\le j\le p}} Q_{j}\cap V=\emptyset$ for any $i$. Let $\tilde f=(f_0,\ldots ,f_n): \C^m\rightarrow \C^{n+1}$ be a reduced representation of $f$. For each hypersurface $Q_i\ (1\le i\le q)$ we denote again by $Q_i$ its defining homogeneous polynomial. Without loss of generality, we suppose that $d=\deg Q_1=\cdots=\deg Q_q$. 
Consider the mapping $\Phi$ from $V$ into $\P^{p-1}(\C)$, which maps a point ${\bf x}=(x_0:\cdots:x_n)\in V$ into the point $\Phi({\bf x})\in\P^{p-1}(\C)$ given by
$$\Phi({\bf x})=(Q_1(x):\cdots : Q_{p}(x)),$$
where $x=(x_0,\ldots,x_n)$. We set $\tilde\Phi(x)=(Q_1(x),\ldots ,Q_{p}(x))$.

Let $Y=\Phi(V)$. Since $V\cap\bigcap_{j=1}^{p}Q_j=\emptyset$, $\Phi$ is a finite morphism on $V$ and $Y$ is a complex projective subvariety of $\P^{p-1}(\C)$ with $\dim Y=k$ and of degree $\delta:=\deg Y\le d^{k}\cdot\deg V.$ 

For every ${\bf a} = (a_1,\ldots,a_p)\in\mathbb Z^p_{\ge 0}$ and ${\bf y} = (y_1,\ldots,y_p)$ we denote ${\bf y}^{\bf a} = y_{1}^{a_{1}}\cdots y_{p}^{a_{p}}$. Let $u$ be a positive integer. We set $\xi_u:=\binom{p+u-1}{u}$ and define
$$ Y_{u}:=\C[y_1,\ldots,y_p]_u/(I_{Y})_u, $$
which is a $\C$-vector space of dimension $H_{Y}(u)$. Put $n_u=H_{Y}(u)-1$ and let $v_0,\ldots,v_{n_u}$ be homogeneous polynomials in $\C[y_1,\ldots,y_p]_u$, whose equivalent classes form a basis of $Y_u$.  

Let $F$ be a meromorphic mapping from $\C^m$ into $\P^{n_u}(\C)$ with the representation:
$$ \tilde F=(v_0(\tilde\Phi\circ \tilde f),\ldots,v_{n_u}(\tilde\Phi\circ \tilde f)).$$
Since $f$ is algebraically nondegenerate, $F$ is linearly nondegenerate.
Then there exists an admissible set $\alpha=(\alpha_0,\ldots,\alpha_{n_u})\in(\mathbb{Z}^m_{\ge 0})^{n_u+1}$ such that
$$W^\alpha(F_0,\ldots,F_{n_u})=\det (D^{\alpha_i}(v_s(\Phi\circ \tilde f)))_{0\le i,s\le n_u}\not\equiv 0.$$

Fix a point $z\in\C^m$ with $Q_i(\tilde f(z))\ne 0$ for all $i=1,\ldots,p$. We define 
$${\bf c}_z = (c_{1,z},\ldots,c_{p,z})\in\mathbb R^{p}_{\ge 0},$$ 
where
$$c_{i,z}:=\log\dfrac{\|\tilde f(z)\|^d\|Q_i\|}{|Q_i(\tilde f)(z)|}\ge 0\text{ for } i=1,\ldots,p.$$
By the definition of the Hilbert weight, there are ${\bf a}_{0,z},\ldots,{\bf a}_{n_u,z}\in\mathbb N^{p}$ with
$$ {\bf a}_{j,z}=(a_{j,1,z},\ldots,a_{j,p,z}),$$ 
where $a_{j,i,z}\in\{1,\ldots,u\},$  such that the residue classes modulo $(I_Y)_u$ of ${\bf y}^{{\bf a}_{0,z}},\ldots,{\bf y}^{{\bf a}_{n_u,z}}$ form a basic of $\C[y_1,\ldots,y_p]_u/(I_Y)_u$ and
\begin{align*}
S_Y(u,{\bf c}_z)=\sum_{j=0}^{n_u}{\bf a}_{j,z}\cdot{\bf c}_z.
\end{align*}
We see that ${\bf y}^{{\bf a}_{j,z}}\in Y_u$ (modulo $(I_Y)_u$). Then we may write
$$ {\bf y}^{{\bf a}_{j,z}}=L_{j,z}(v_0,\ldots ,v_{n_u}), $$ 
where $L_{j,z}\ (0\le j\le n_u)$ are independent linear forms.
We have
\begin{align*}
\log\prod_{j=0}^{n_u} |L_{j,z}(\tilde F(z))|&=\log\prod_{j=0}^{n_u}\prod_{i=1}^p|Q_i(\tilde f(z))|^{a_{j,i,z}}\\
&=-S_Y(u,{\bf c}_z)+du(n_u+1)\log \|\tilde f(z)\| +O(u(n_u+1)).
\end{align*}
Therefore,
\begin{align}\label{3.3}
S_Y(u,{\bf c}_z) = \log\prod_{j=0}^{n_u}\dfrac{1}{|L_{j,z}(\tilde F(z))|} + du(n_u+1)\log \|\tilde f(z)\|+O(u(n_u+1)).
\end{align}
From Theorem \ref{2.3} we have
\begin{align}\label{3.4}
\dfrac{1}{u(n_u+1)}S_Y(u,{\bf c}_z)\ge&\dfrac{1}{(k+1)\delta}e_Y({\bf c}_z)-\dfrac{(2k+1)\delta}{u}\max_{1\le i\le p}c_{i,z}.
\end{align}
Combining (\ref{3.3}), (\ref{3.4}), we get
\begin{align}\nonumber
\dfrac{1}{(k+1)\delta}e_Y({\bf c}_z)\le & \dfrac{1}{u(n_u+1)}\log\prod_{j=0}^{n_u}\dfrac{1}{|L_{j,z}(\tilde F(z))|} + d\log \|\tilde f(z)\| \\
\label{3.5}
\begin{split}
&+\dfrac{(2k+1)\delta}{u}\sum_{i=1}^p\log\dfrac{\|\tilde f(z)\|^d\|Q_i\|}{|Q_i(\tilde f)(z)|}+O(1/u).
\end{split}
\end{align}
Suppose that $c_{1,z}\ge c_{2,z}\ge\cdots\ge c_{p,z}$ and denote by $t$ the smallest index such that $V\cap\bigcap_{j=1}^tQ_j=\emptyset$. Note that $\Delta_{\{Q_j\}_{j=1}^t,V}\le\Delta_{\mathcal P,V}$. Then by Lemma \ref{2.4}, we have
\begin{align}\label{3.6}
\begin{split}
e_Y({\bf c}_z)&\ge\frac{\delta}{\Delta_{\mathcal P,V}}(c_{1,z}+\cdots +c_{t,z}) =\frac{\delta}{\Delta_{\mathcal P,V}}\left (\sum_{i=1}^t\log\dfrac{\|\tilde f(z)\|^d\|Q_i\|}{|Q_i(\tilde f)(z)|}\right)\\
&=\frac{\delta}{\Delta_{\mathcal P,V}}\left (\sum_{i=1}^p\log\dfrac{\|\tilde f(z)\|^d\|Q_i\|}{|Q_i(\tilde f)(z)|}\right)+O(1).
\end{split}
\end{align}
Then, from (\ref{3.5}) and (\ref{3.6}) we have
\begin{align}\label{3.7}
\begin{split}
\dfrac{1}{\Delta_{\mathcal P,V}}\log\prod_{i=1}^p\dfrac{\|\tilde f (z)\|^d}{|Q_i(\tilde f)(z)|}\le & \dfrac{k+1}{u(n_u+1)}\log\prod_{j=0}^{n_u}\dfrac{1}{|L_{j,z}(\tilde F(z))|}+d(k+1)\log \|\tilde f(z)\|\\
&+\dfrac{(2k+1)(k+1)\delta}{u}\sum_{i=1}^p\log\dfrac{\|\tilde f(z)\|^d\|Q_i\|}{|Q_i(\tilde f)(z)|}+O(1),
\end{split}
\end{align}
where the term $O(1)$ does not depend on $z$. 

Set $m_0=(2k+1)(k+1)\delta, x=\frac{1}{\Delta_{\mathcal P,V}}-\frac{m_0}{u}$ and $b=\dfrac{k+1}{u(n_u+1)}$. From the above inequality and Proposition \ref{2.2}(iii), we get
\begin{align}\label{3.10new}
\begin{split}
\log \dfrac{\|\tilde f (z)\|^{xdp-d(k+1)}|W^\alpha(\tilde{F}(z))|^b}{\prod_{i=1}^p|Q_i(\tilde f)(z)|^{x}}\le b \log\dfrac{|W^\alpha(L_{0,z}(\tilde F(z)),\ldots,L_{n_u,z}(\tilde F(z)))|}{\prod_{j=0}^{n_u}|L_{j,z}(\tilde F(z))|}+O(1).
\end{split}
\end{align}
Here, we note that $L_{j,z}$ depends on $j$ and $z$, but the number of these linear forms is finite (at most $\xi_u$). We denote by $\mathcal L$ the set of all $L_{j,z}$ occurring in the above inequalities. The inequality (\ref{3.10new}) implies that
$$ \log \dfrac{\|\tilde f (z)\|^{xdp-d(k+1)}|W^\alpha(\tilde{F}(z))|^b}{\prod_{i=1}^p|Q_i(\tilde f)(z)|^{x}}\le b \max_{\mathcal J\subset\mathcal L}\log^+\dfrac{|W^\alpha(\tilde{F}(z))|}{\prod_{L\in\mathcal J}|L(\tilde F(z))|}+O(1)$$
for every $z\in\C^m\setminus\bigcup_{j=1}^p(Q_j(\tilde f))^{-1}\{0\}$, where the maximum $\max_{\mathcal J\subset\mathcal L}$ is taken over all subsets $\mathcal J$ of $\mathcal L$ such that $\{L;\ L\in\mathcal J\}$ is independent and $\sharp\mathcal J=n_u+1$. Integrating both sides of the above inequality and using lemma on logarithmic derivative (Lemma \ref{2.1}) for the right hand side, we obtain
\begin{align}\label{3.9}
\biggl\|\ \left(p-\frac{k+1}{x}\right)T_f(r)\le\sum_{j=1}^p\frac{1}{d}N_{Q_j(\tilde f)}(r)-\frac{b}{dx}N_{W^\alpha(\tilde{F})}(r)+o(T_f(r)).
\end{align}
We now estimate the quantity $N_{W^\alpha(\tilde{F})}(r)$. Let $z\in\C^m$ which is outside the indeterminacy locus of $f$. We set $c_{i}=\max\{0,\nu^0_{Q_i( f)}(z)-n_u\}\ (1\le i\le p)$ and 
$${\bf c}=(c_{1},\ldots,c_{p})\in\mathbb Z^p_{\ge 0}.$$
Without loss of generality, we may suppose that $c_1\ge c_2\ge\cdots\ge c_p$ and let $t$ be the smallest index such that $V\cap\bigcap_{j=1}^tQ_j=\emptyset$ as above. Then, it is clear that $c_{i}=0$ for all $i\ge t$. By the definition of the Hilbert weight, there are
$${\bf a}_i=(a_{i,1},\ldots,a_{i,p}),0\le i\le n_u, a_{i,s}\in\{1,...,u\}$$
such that ${\bf y}^{{\bf a}_0},...,{\bf y}^{{\bf a}_{n_u}}$ is a basic of $\C[y_1,\ldots,y_p]_u/(I_{Y})_u$ and
$$ S_{Y}(u,{\bf c})=\sum_{i=0}^{n_u}{\bf a}_i\cdot{\bf c}.$$
Similarly as above, we write ${\bf y}^{{\bf a}_i}=L_i(v_0,...,v_{n_u})$, where $L_0,...,L_{n_u}$ are linearly independent linear forms. We see that $W^\alpha(\tilde F)=cW^{\alpha}(L_0(\tilde F),\ldots,L_{n_u}(\tilde F))$ for a nonzero constant $c$. Therefore,
\begin{align*}
\nu_{W^\alpha(\tilde F)}(z)&=\nu_{W^{\alpha}(L_0(\tilde F),\ldots,L_{n_u}(\tilde F))}(z)\\
&\ge\sum_{i=0}^{n_u}\left(\nu_{L_i(\tilde F)}(z)-\min\{n_u,\nu_{L_i(\tilde F)}(z)\}\right)\\
&\ge\sum_{i=0}^{n_u}\sum_{j=1}^pa_{i,j}\left(\nu_{Q_j(\tilde f)}(z)-\min\{n_u,\nu_{Q_j(\tilde f)}(z)\}\right)\\
&=\sum_{i=1}^{n_u}{\bf a}_i\cdot{\bf c}=S_Y(u,{\bf c}).
\end{align*}
Also, by Lemma \ref{2.4} we have
$$e_{Y}({\bf c})\ge\frac{\delta}{\Delta_{\mathcal P,V}}(c_1+\cdots +c_t)=\frac{\delta}{\Delta_{\mathcal P,V}}(c_1+\cdots +c_p).$$
On the other hand, by Theorem \ref{2.3} we have that 
\begin{align*}
 S_{Y}(u,{\bf c}) &\ge\frac{u(n_u+1)}{(k+1)\delta}e_Y({\bf c})-(2k+1)\delta(n_u+1)\max_{1\le i\le p}c_{i}\\
&\ge \left (\frac{u(n_u+1)}{(k+1)\Delta_{\mathcal P,V}}-(2k+1)\delta (n_u+1)\right)\sum_{j=1}^{p}\max\{0,\nu^0_{Q_j(\tilde f)}(z)-n_u\}.
\end{align*}
Therefore,
\begin{align*}
\frac{b}{x}\nu_{W^\alpha(\tilde F)}(z)&\ge\frac{1}{x}\left(\frac{1}{\Delta_{\mathcal P,V}}-\frac{m_0}{u}\right)\sum_{j=1}^{p}\max\{0,\nu^0_{Q_j(\tilde f)}(z)-n_u\}\\
&=\sum_{j=1}^{p}\max\{0,\nu^0_{Q_j(\tilde f)}(z)-n_u\}=\sum_{j=1}^{p}\left(\nu^0_{Q_j(\tilde f)}(z)-\min\{n_u,\nu^0_{Q_j(\tilde f)}(z)\}\right).
\end{align*}
Since this inequality holds for all $z$ outside the indeterminacy locus of $f$, it yields that
$$ \frac{b}{x}N_{W^\alpha(\tilde F)}(r)\ge\sum_{j=1}^{p}\left(N_{Q_j(\tilde f)}(r)-N^{[n_u]}_{Q_j(\tilde f)}(r)\right).$$
Combining this inequality with (\ref{3.9}), we obtain that
\begin{align}\label{3.10}
\biggl\|\ \left(p-\frac{k+1}{x}\right)T_f(r)\le\sum_{j=1}^p\frac{1}{d}N^{[n_u]}_{Q_j(\tilde f)}(r)+o(T_f(r)).
\end{align}
Choose $u=\lceil\Delta_{\mathcal P,V}m_0(\Delta_{\mathcal P,V}(k+1)+\epsilon)\epsilon^{-1}\rceil$. Then $\frac{\Delta_{\mathcal P,V}m_0}{u}\le\frac{\epsilon}{\Delta_{\mathcal P,V}(k+1)+\epsilon}=1-\frac{\Delta_{\mathcal P,V}(k+1)}{\Delta_{\mathcal P,V}(k+1)+\epsilon}$. Therefore
$$\frac{k+1}{x}=\frac{\Delta_{\mathcal P,V}(k+1)}{1-\frac{\Delta_{\mathcal P,V}m_0}{u}}\le \Delta_{\mathcal P,V}(k+1)+\epsilon.$$
We also have $n_u=H_Y(u)-1\le\delta\binom{k+u}{k}-1\le d^k\deg(V)\binom{k+u}{k}-1.$ Hence,\\
$\bullet$ if $k=1$ then
\begin{align*}
n_u+1&\le  d\deg(V)(1+u)\\
&< d\deg(V) \left(\Delta_{\mathcal P,V}6d\deg(V)(2\Delta_{\mathcal P,V}\epsilon^{-1}+1)+2\right)\\ 
&< d^2\deg(V)^2 e\Delta_{\mathcal P,V}7(2\Delta_{\mathcal P,V}\epsilon^{-1}+1);
\end{align*}\\
$\bullet$ if $k\ge 2$ then
\begin{align*}
n_u+1&\le d^k\deg (V)e^{k}\left(1+\dfrac{u}{k}\right)^{k}\\
&<d^k\deg (V)e^k\left(1+\dfrac{d^k\deg (V)\Delta_{\mathcal P,V}(2k+1)(k+1)(\Delta_{\mathcal P,V}(k+1)\epsilon^{-1}+1)+1}{k}\right)^k\\
&<d^{k^2+k}\deg(V)^{k+1}e^k\Delta_{\mathcal P,V}^k(2k+5)^k(\Delta_{\mathcal P,V}(k+1)\epsilon^{-1}+1)^k.
\end{align*} 
Thus, we always get
\begin{align*}
n_u+2&\le \lceil d^{k^2+k}\deg(V)^{k+1}e^k\Delta_{\mathcal P,V}^k(2k+5)^k(\Delta_{\mathcal P,V}(k+1)\epsilon^{-1}+1)^k\rceil\\
&\le\lceil d^{k^2+k}\deg(V)^{k+1}e^k(2k+5)^k(\lambda^2(k+1)\epsilon^{-1}+\lambda)^k\rceil=L+2.
\end{align*}
Therefore, from (\ref{3.10}) and the above estimate we have
\begin{align}\label{3.11}
\bigl\|\ (p-\Delta_{\mathcal P,V}(k+1)-\epsilon)T_f(r)\le\sum_{j=1}^p\frac{1}{d}N^{[L]}_{Q_j(\tilde f)}(r)+o(T_f(r)).
\end{align}
We remark that $\Delta_{\mathcal P,V}\le\lambda$ and
$$p-\Delta_{\mathcal P,V}(k+1)-\epsilon\ge q-\left[\lambda+\frac{1}{2}\right]-\lambda(k+1)-\epsilon.$$
Then, the inequality (\ref{3.11}) implies that
$$\bigl\|\ \left(q-\left[\lambda+\frac{1}{2}\right]-\lambda(k+1)-\epsilon\right)T_f(r)\le\sum_{j=1}^q\frac{1}{d}N^{[L]}_{Q_j(\tilde f)}(r)+o(T_f(r)).$$
The theorem is proved.
\end{proof}

\begin{remark}{\rm  
In the above proof, we may also set $\mathcal P=\mathcal Q$, whose the distributive constant with respect to $V$ does not exceed $(N-k+1)=2\lambda-1$. Then we obtain again the second main theorem of Quang \cite[Theorem 1.1]{Q19} in the form
$$\bigl\|\ (q-(2\lambda-1)(k+1)-\epsilon)T_f(r)\le\sum_{j=1}^q\frac{1}{\deg Q_j}N^{[L]}_{Q_j(\tilde f)}(r)+o(T_f(r)),$$
with a better truncation level 
$$L=\lceil d^{k^2+k}\deg(V)^{k+1}e^k(2k+5)^k((2\lambda-1)^2(k+1)\epsilon^{-1}+(2\lambda-1))^k-2\rceil.$$}
\end{remark}

\noindent
{\bf Disclosure statement.} No potential conflict of interest was reported by the authors.

\end{document}